\documentclass{cmslatex}

\usepackage[paperwidth=7in, paperheight=10in, margin=.875in]{geometry}
\usepackage[backref,colorlinks,linkcolor=red,anchorcolor=green,citecolor=blue]{hyperref}
\usepackage{amsfonts,amssymb}
\usepackage{amsmath}
\usepackage{graphicx}
\usepackage{cite}
\usepackage{enumerate}
\usepackage{color}
\usepackage{bm}
\renewcommand{\theequation}{\arabic{section}.\arabic{equation}}

\allowdisplaybreaks
\begin{document}

\title{Global existence of ideal invicid compressible and heat conductive fluids with radial symmetry.}

\author{Peng Lu\thanks{School of Mathematics Science, Fudan University, Shanghai, P. R. China (plu17@fudan.edu.cn).}
\and Yi Zhou\thanks{School of Mathematics Science, Fudan University, Shanghai, P. R. China (yizhou@fudan.edu.cn).}}

\pagestyle{myheadings} \markboth{Global existence of ideal invicid compressible and heat conductive fluids with radial symmetry.}{Global existence of ideal invicid compressible and heat conductive fluids with radial symmetry.}
\maketitle

\begin{abstract}
In this paper, we study the global existence of classical solutions to the three dimensional ideal invicid compressible and heat conductive fluids with radial symmetrical data in $H^s(\mathbb{R}^3)$. Our proof is based on the symmetric hyperbolic structure of the system.
\end{abstract}

\section{Introduction}\label{intro}

The motion for a compressible viscous, heat-conductive, isotropic Newtonian fluid is described by the system of equations
\begin{equation}\label{1.1}
\left\{
\begin{aligned}
& \rho_t+\nabla\cdot(\rho\bm{u})=0 \\
& \rho\bm{u}_t+\rho\bm{u}\cdot\nabla\bm{u}+\nabla p-\Big(\mu'-\frac{2}{3}\mu\Big)\nabla(\nabla\cdot\bm{u})-\nabla\cdot\big(\mu(\nabla\bm{u}+(\nabla\bm{u})^T)\big)=0 \\
& \Big(\rho\Big(\frac{|\bm{u}|^2}{2}+e\Big)\Big)_t+\nabla\cdot\Big(\rho\bm{u}\Big(\frac{|\bm{u}|^2}{2}+e\Big)+p\bm{u}\Big)-\nabla\cdot\Big(\mu(\nabla\bm{u}+(\nabla\bm{u})^T)\bm{u} \\
& +\Big(\mu'-\frac{2}{3}\mu\Big)\bm{u}(\nabla\cdot\bm{u})\Big)=\nabla\cdot(\kappa\nabla T),
\end{aligned}
\right.
\end{equation}

where $t\geqslant 0$, $x=(x_1,x_2,x_3)\in\mathbb{R}^3$. $\rho>0$ denotes the density, $\bm{u}=(u_1,u_2,u_3)$ the fluid velocity, $T>0$ the absolute temperature, $e$ denotes the internal energy, and $p$ denotes the pressure. And the positive constants $\mu,\mu'$satisfying
$$\mu>0,\quad \mu'+\frac{2\mu}{3}\geqslant 0$$
describe the viscosity.

There are very rich results about compressible Navier-Stokes system, such as small classical solutions with finite energy by Matsumura-Nishida \cite{Matsumura-Nishida}, see also Huang-Li \cite{Huang-Li} about the case of vacuum, weak, finite-energy solutions by Lions \cite{Lions}, variational solutions by Feireisl \cite{Feireisl} and Feireisl-Novotn$\acute{y}$-Petzeltov$\acute{a}$ \cite{Feireisl-Novotny-Petzeltov} , solutions in Besov spaces with the interpolation index one by Chikami-Danchin \cite{Chikami-Danchin}, Danchin \cite{Danchin}, self-similar solutions by Guo-Jiang \cite{Guo-Jiang}, Li-Chen-Xie \cite{Li-Chen-Xie} (density-dependent viscosity) and Germain-Iwabuchi \cite{Germain-Iwabuchi}.

There are also some literature related to the vacuum. Xin \cite{Xin} proved the non-existence of smooth solutions for the initial density with the compact support. Hoff and Smoller \cite{Hoff-Smoller} considered 1-D barotropic Navier-Stokes equations and showed that the persistency of the almost everywhere positivity of the density can prevent the formulation of vacuum state. Jang and Masmoudi \cite{Jang-Masmoudi} obtained local solutions of the 3D compressible Euler equations under the barotropic condition with a physical vacuum, see also \cite{Jang-Masmoudi2} about problems of vacuum state. Recently, Lai, Liu and Tarfulea \cite{Lai-Liu-Tarfulea} studied the derivation of some non-isothermal hydrodynamic models (including non-isothermal ideal gas) and established the corresponding maximum principle.

In the classical paper of Matsumura-Nishida \cite{Matsumura-Nishida}, they proved the global existence of classical solutions with small data of $O(\varepsilon)$ in $H^s$, where $\varepsilon$ depends on $\mu$, $\mu'$ and $\kappa$. The main purpose of this paper is to improve the result of \cite{Matsumura-Nishida} in radial symmetry case. In this case we can take the small constant $\varepsilon$ independent of $\mu$ and $\mu'$, and only depends on $\kappa$. More precisely, we can set $\mu=\mu'=0$ and \eqref{1.1} will thus reduce to the following system
\begin{equation}\label{Heat-Hydro}
\left\{
\begin{aligned}
& \rho_t+\nabla\cdot(\rho\bm{u})=0 \\
& (\rho\bm{u})_t+\nabla\cdot(\rho\bm{u}\otimes\bm{u})+\nabla p=0 \\
& \left(\rho\left(\frac{1}{2}|\bm{u}|^2+e\right)\right)_t+\nabla\cdot\left(\rho\bm{u}\left(\frac{1}{2}|\bm{u}|^2+e\right)+p\bm{u}\right)=\nabla\cdot(\kappa\nabla T),
\end{aligned}
\right.
\end{equation}

We assume the following conditions on \eqref{Heat-Hydro}:
\begin{itemize}
\item[1.] The gas is ideal : $p=RT\rho$, where $R$ is a positive constant;
\item[2.] The gas is polytropic : $e=c_VT$, where $c_V$ is a positive constant which denotes the specific heat at constant volume.
\end{itemize}

Assume that the positive constants $R, ~ c_V, ~ \kappa=1$, then the system \eqref{Heat-Hydro} can be written in the following form
\begin{equation}\label{Heat-Hydro-2}
\left\{
\begin{aligned}
& \rho_t+\nabla\cdot(\rho\bm{u})=0 \\
& \bm{u}_t+\bm{u}\cdot\nabla\bm{u}+\frac{1}{\rho}\nabla(\rho T)=0 \\
& T_t+\bm{u}\cdot\nabla T+T(\nabla\cdot\bm{u})=\frac{\Delta T}{\rho}.
\end{aligned}
\right.
\end{equation}

Suppose that the initial data
\begin{equation}\label{data small}
\rho(0,x)=1+a_0(r),\quad T(0,x)=1+\theta_0(r),\quad\bm{u}(0,x)=\bm{u}_0(r)=u_0(r)\bm{\omega}
\end{equation}

satisfy
$$\|a_0\|_{H^s}^2+\|\bm{u}_0\|_{H^s}^2+\|\theta_0\|_{H^s}^2\leqslant\varepsilon^2,$$

where $s>5$ is an integer, $r=|x|$ and $\bm{\omega}=\frac{x}{|x|}$, and $\varepsilon>0$ is a small constant.

By the uniqueness of classical solutions, the solutions must have the following form
$$\rho=1+a(t,r),\quad T=1+\theta(t,r),\quad \bm{u}=u(t,r)\bm{\omega},$$

as a result, we obtain
$$\nabla\times\bm{u}\equiv 0.$$

So we may consider the follow system.
\begin{equation}\label{Heat-Hydro-3}
\left\{
\begin{aligned}
& a_t+\bm{u}\cdot\nabla a+(1+a)(\nabla\cdot\bm{u})=0 \\
& \bm{u}_t+\bm{u}\cdot\nabla\bm{u}+\nabla\theta+\frac{1+\theta}{1+a}\nabla a=0 \\
& \theta_t+\bm{u}\cdot\nabla\theta+(1+\theta)(\nabla\cdot\bm{u})=\frac{\Delta\theta}{1+a}
\end{aligned}
\right.
\end{equation}
with the condition
\begin{equation}\label{irrotational}
\nabla\times\bm{u}\equiv 0.
\end{equation}

Our main result can be stated as follows.
\begin{theorem}\label{main}
Consider the Cauchy problem of the three dimensional system \eqref{Heat-Hydro-2}-\eqref{irrotational} ( or \eqref{Heat-Hydro-3}-\eqref{irrotational}) with data \eqref{data small}. Then there exists a constant $\varepsilon_0>0$ such that for $\forall ~ \varepsilon<\varepsilon_0$, the system \eqref{Heat-Hydro-2}-\eqref{irrotational} ( or \eqref{Heat-Hydro-3}-\eqref{irrotational}) admits a global solution
$$(a,\bm{u})\in L^\infty(\mathbb{R}_+;H^s(\mathbb{R}^3))\cap L^2(\mathbb{R}_+;H^s(\mathbb{R}^3)),$$

and
$$\theta\in L^\infty(\mathbb{R}_+;H^s(\mathbb{R}^3))\cap L^2(\mathbb{R}_+;H^{s+1}(\mathbb{R}^3)).$$
\end{theorem}

As Thm \ref{main} shows, heat conduction effect alone can prevent the formation of shock despite the lack of viscosity.

\begin{remark}
It's clear that the solution of \eqref{Heat-Hydro} have the following conservation laws
$$\frac{d}{dt}\int a dx\equiv 0, \qquad \frac{d}{dt}\int (a+1)\bm{u}dx\equiv 0,$$

and
$$\frac{d}{dt}\int\left(\frac{|\bm{u}|^2}{2}+\frac{a|\bm{u}|^2}{2}+a\theta+\theta\right)dx\equiv 0.$$
\end{remark}

We set
$$
\begin{aligned}
E_{k,1}(t) & \triangleq\sum\limits_{|\alpha|\leqslant k}\sup\limits_{\tau\in[0,t]}\big(\|\partial^\alpha a(\tau)\|_{L^2}^2+\|\partial^\alpha\bm{u}(\tau)\|_{L^2}^2+\|\partial^\alpha \theta(\tau)\|_{L^2}^2\big) \\
& +\sum\limits_{|\alpha|\leqslant k}\int_0^t\|\nabla\partial^\alpha\theta(\tau)\|_{L^2}^2d\tau
\end{aligned}
$$
for $0\leqslant k\leqslant s$, and

$$E_{k,2}(t)\triangleq\sum\limits_{|\alpha|\leqslant k-1}\int_0^t\left(\|\nabla\partial^\alpha a(\tau)\|_{L^2}^2+\|\nabla\partial^\alpha\bm{u}(\tau)\|_{L^2}^2\right)d\tau$$
for $1\leqslant k\leqslant s$, where
$$\partial=(\partial_t,\partial_{x_1},\partial_{x_2},\partial_{x_3}).$$

According to \eqref{data small}, it's clear that $\exists ~ M>0$ such that
$$E_{s,1}(0)+E_{s,2}(0)\leqslant M^2\varepsilon^2.$$

Due to the local existence result, there exists a positive time $t_*\leqslant +\infty$ such that
\begin{equation}\label{energy small assumption}
t_*=\max\big\{t\geqslant 0 ~ \big| ~ E_{s,1}(\tau)+E_{s,2}(\tau)\leqslant\varepsilon, ~ \forall ~ \tau\in[0,t_*)\big\}.
\end{equation}

We have the following lemma.
\begin{lemma}\label{entropy S lemma}
Let
\begin{equation}\label{entropy S}
S=\ln\left(\frac{T}{\rho}\right)=\ln\left(\frac{1+\theta}{1+a}\right)
\end{equation}
denotes the entropy of unit mass, then the entropy of the system increases.
\end{lemma}
\begin{proof}
The entropy of unit volume is
$$\rho S=\rho\ln\left(\frac{T}{\rho}\right)=\rho\ln T-\rho\ln\rho,$$

and we can establish the evolution equation of $\rho S$:
$$\partial_t(\rho S)=\rho_tS+\rho S_t=-S\nabla\cdot(\rho\bm{u})+\nabla\cdot(\rho\bm{u})+\frac{\rho T_t}{T},$$

then the third equation of \eqref{Heat-Hydro-2} gives the result
$$
\begin{aligned}
\frac{d}{dt}\int\rho Sdx & =\int\rho\bm{u}\cdot\nabla S+\frac{\rho T_t}{T}dx \\
& =\int\frac{\rho\bm{u}\cdot\nabla T}{T}-\bm{u}\cdot\nabla\rho+\frac{\rho}{T}\left(\frac{\Delta T}{\rho}-\bm{u}\cdot\nabla T-T(\nabla\cdot\bm{u})\right)dx \\
& =\int\frac{\Delta T}{T}dx=\int\frac{|\nabla T|^2}{T^2}dx\geqslant 0.
\end{aligned}
$$

This completes the proof of Lemma \ref{entropy S lemma}.
\end{proof}

\section{Basic Energy Estimate}

By Lemma \ref{entropy S lemma}, we have
\begin{equation}\label{Entropy increase}
\frac{d}{dt}\int(1+a)\ln\left(\frac{1+a}{1+\theta}\right)dx+\int\frac{|\nabla\theta|^2}{(1+\theta)^2}dx=0.
\end{equation}

Making linear combination of \eqref{Entropy increase} and the conservation quantities, we obtain
\begin{equation}\label{Basic L2 estimate}
\frac{d}{dt}\int(1+a)\ln\left(\frac{1+a}{1+\theta}\right)+\left(\frac{|\bm{u}|^2}{2}+\frac{a|\bm{u}|^2}{2}-a+a\theta+\theta\right)dx+\int\frac{|\nabla\theta|^2}{(1+\theta)^2}dx=0.
\end{equation}

Making a Taylor expansion of \eqref{entropy S} with respect to $a$ and $\theta$, we have
$$\ln\left(\frac{1+a}{1+\theta}\right)=a-\theta-\frac{a^2}{2}+\frac{\theta^2}{2}+r(a,\theta),$$

where the remainder $r(a,\theta)$ satisfies
$$r(a,\theta)=O(a^3+\theta^3),\quad |a|+|\theta|\to 0.$$

Go back to \eqref{Basic L2 estimate}, we get
$$
\begin{aligned}
& \|a(t)\|_{L^2}^2+\|\bm{u}(t)\|_{L^2}^2+\|\theta(t)\|_{L^2}^2+2\int_0^t\|\nabla\theta(\tau)\|_{L^2}^2d\tau \\
= & ~ E_{0,1}(0)+\int a_0(\theta_0^2+|\bm{u}_0|^2-a_0^2)+(1+a_0)r(a_0,\theta_0)dx \\
- & ~ \int a(\theta^2+|\bm{u}|^2-a^2)+(1+a)r(a,\theta)dx+2\int_0^t\int\theta|\nabla\theta|^2\frac{(2+\theta)}{(1+\theta)^2}dxd\tau.
\end{aligned}
$$

By \eqref{energy small assumption}, we have
$$\|a\|_{L^\infty}+\|\theta\|_{L^\infty}\leqslant C\sqrt{E_{2,1}(t)}\leqslant C\varepsilon^{\frac{1}{2}},$$

this gives the result
\begin{equation}\label{Key Energy estimate 0}
E_{0,1}(t)\lesssim E_{0,1}(0)+E_{2,1}^{3/2}(0)+E_{2,1}^{3/2}(t),
\end{equation}

here and hereafter $A\lesssim B$ means $A\leqslant CB$ with a positive constant $C$.

\section{The Estimate of $E_k$}

Firstly, we write the equations of $a$ and $u$ in \eqref{Heat-Hydro-3} in the following form of symmetric hyperbolic systems.
\begin{equation}\label{Heat-Hydro-Hyperbolic system}
A_0(\bm{U},\theta)\bm{U}_t+\sum\limits_{j=1}^3A_j(\bm{U},\theta)\partial_j\bm{U}+\bm{F}=0,
\end{equation}

where
$$
\bm{U}=\left(
\begin{array}{c}
a \\
u_1 \\
u_2 \\
u_3
\end{array}
\right),\quad
A_0=\left(
\begin{array}{cccc}
\frac{1+\theta}{1+a} & 0 & 0 & 0  \\
0 & 1+a & 0 & 0  \\
0 & 0 & 1+a & 0  \\
0 & 0 & 0 & 1+a  \\
\end{array}
\right),
$$

and
$$
A_j=\left(
\begin{array}{cccc}
\frac{1+\theta}{1+a}u_j & (1+\theta)\delta_{1j} & (1+\theta)\delta_{2j} & (1+\theta)\delta_{3j}  \\
(1+\theta)\delta_{1j} & (1+a)u_j & 0 & 0 \\
(1+\theta)\delta_{2j} & 0 & (1+a)u_j & 0 \\
(1+\theta)\delta_{3j} & 0 & 0 & (1+a)u_j  \\
\end{array}
\right),\quad
\bm{F}=\left(
\begin{array}{c}
0 \\
(1+a)\theta_{x_1} \\
(1+a)\theta_{x_2} \\
(1+a)\theta_{x_3}
\end{array}
\right).
$$

By applying $\partial^\alpha$ to \eqref{Heat-Hydro-Hyperbolic system}, where the multi-index $\alpha$ satisfying $0<|\alpha|\leqslant k$, and the positive integer $k\leqslant s$, we obtain
$$A_0\partial^\alpha\bm{U}_t+\sum\limits_{j=1}^3A_j\partial_j\partial^\alpha\bm{U}=\left(A_0\partial^\alpha\bm{U}_t-\partial^\alpha(A_0\bm{U}_t)\right)+\sum\limits_{j=1}^3\big( A_j\partial_j\partial^\alpha\bm{U}-\partial^\alpha\left(A_j\partial_j\bm{U}\right)\big)-\partial^\alpha\bm{F}.$$

Then we take the $L^2$ inner product of the above equation with $\partial^\alpha\bm{U}$ and integrate with respect to $t$. By the symmetry of $A_j$ and $A_0$, we have the following energy estimate
\begin{equation}\label{Key Energy estimate k-1-1}
\begin{aligned}
& \int\partial^\alpha\bm{U}^TA_0\partial^\alpha\bm{U}dx \\
= & ~ 2\int_0^t\int\partial^\alpha\bm{U}^T\Big(\left(A_0\partial^\alpha\bm{U}_t-\partial^\alpha(A_0\bm{U}_t)\right)+\sum\limits_{j=1}^3\big(A_j\partial_j\partial^\alpha\bm{U}- \partial^\alpha\left(A_j\partial_j\bm{U}\right)\big)\Big)dxd\tau \\
+ & \int\partial^\alpha\bm{U}_0^TA_0(\bm{U}_0,\theta_0)\partial^\alpha\bm{U}_0dx+\int_0^t\int\partial^\alpha\bm{U}^T\left(\partial_tA_0+\sum\limits_{j=1}^3\partial_jA_j\right) \partial^\alpha\bm{U}dxd\tau \\
- & ~ 2\int_0^t\int \partial^\alpha\bm{F}\cdot\partial^\alpha\bm{U}dxd\tau,
\end{aligned}
\end{equation}

where
$$-2\int_0^t\int\partial^\alpha\bm{F}\cdot\partial^\alpha\bm{U}dxd\tau=-2\int_0^t\int\partial^\alpha\bm{u}\cdot\nabla\partial^\alpha\theta+\partial^\alpha\bm{u}\cdot\partial^\alpha (a\nabla\theta)dxd\tau.$$

On the other hand, we make energy estimate of $\theta$ to obtain
\begin{equation}\label{Key Energy estimate k-1-2}
\begin{aligned}
& \int|\partial^\alpha\theta|^2dx+2\int_0^t\int|\nabla\partial^\alpha\theta|^2dxd\tau \\
= & ~ \int|\partial^\alpha\theta_0|^2dx+2\int_0^t\int\partial^\alpha\bm{u}\cdot\nabla\partial^\alpha\theta dx+2\int_0^t\int\partial^\alpha(\theta\bm{u})\cdot\nabla\partial^\alpha \theta dxd\tau \\
+ & ~ 2\int_0^t\int\frac{\partial^\alpha\theta\nabla\partial^\alpha\theta\cdot\nabla a}{(1+a)^2}+\frac{a|\nabla\partial^\alpha\theta|^2}{1+a}+\partial^\alpha\theta\left( \partial^\alpha\left(\frac{\Delta\theta}{1+a}\right)-\frac{\Delta\partial^\alpha\theta}{1+a}\right)dxd\tau. \\
\end{aligned}
\end{equation}

Adding \eqref{Key Energy estimate k-1-1} to \eqref{Key Energy estimate k-1-2}, we get
$$
\begin{aligned}
& \int\partial^\alpha\bm{U}^TA_0\partial^\alpha\bm{U}dx+\int|\partial^\alpha\theta|^2dx+2\int_0^t\int|\nabla\partial^\alpha\theta|^2dxd\tau \\
= & ~ \int\partial^\alpha\bm{U}_0^TA_0(\bm{U}_0,\theta_0)\partial^\alpha\bm{U}_0dx+\int|\partial^\alpha\theta_0|^2dx \\
+ & ~ 2\int_0^t\int\partial^\alpha(\theta\bm{u})\cdot\nabla\partial^\alpha\theta-\partial^\alpha\bm{u}\cdot\partial^\alpha(a\nabla\theta)dxd\tau \\
+ & ~ 2\int_0^t\int\frac{\partial^\alpha\theta\nabla\partial^\alpha\theta\cdot\nabla a}{(1+a)^2}+\frac{a|\nabla\partial^\alpha\theta|^2}{1+a}+\partial^\alpha\theta\left( \partial^\alpha\left(\frac{\Delta\theta}{1+a}\right)-\frac{\Delta\partial^\alpha\theta}{1+a}\right)dxd\tau \\
+ & ~ 2\int_0^t\int\partial^\alpha\bm{U}^T\Big(\left(A_0\partial^\alpha\bm{U}_t-\partial^\alpha(A_0\bm{U}_t)\right)+\sum\limits_{j=1}^3\big(A_j\partial_j\partial^\alpha\bm{U} -\partial^\alpha\left(A_j\partial_j\bm{U}\right)\big)\Big)dxd\tau \\
+ & ~ \int_0^t\int\partial^\alpha\bm{U}^T\left(\partial_tA_0+\sum\limits_{j=1}^3\partial_jA_j\right)\partial^\alpha\bm{U}dxd\tau.
\end{aligned}
$$

We have the following lemma from \cite{Li-Zhou} to deal with the nonlinear terms.
\begin{lemma}\label{nonlinear lemma}
For $\forall ~ N\in\mathbb{N}_+$, we have
$$
\begin{aligned}
\|fg\|_{H^N} & \lesssim\bigg(\sum\limits_{|\alpha_1|\leqslant\lfloor\frac{N-1}{2}\rfloor}\|\partial^{\alpha_1}f\|_{L^\infty}\bigg)\bigg(\sum\limits_{|\alpha_3|\leqslant N}\| \partial^{\alpha_3}g\|_{L^2}\bigg) \\
& +\bigg(\sum\limits_{|\alpha_2|\leqslant\lfloor\frac{N-1}{2}\rfloor}\|\partial^{\alpha_2}g\|_{L^\infty}\bigg)\bigg(\sum\limits_{|\alpha_4|\leqslant N}\|\partial^{\alpha_4}f \|_{L^2}\bigg).
\end{aligned}
$$

For any multi-index $\beta$ satisfying $|\beta|=N>0$, we have
$$
\begin{aligned}
\big\|\partial^\beta(fg)-f\partial^\beta g\big\|_{L^2} & \lesssim\bigg(\sum\limits_{|\beta_1|\leqslant\lfloor\frac{N}{2}\rfloor}\|\partial^{\beta_1}f\|_{L^\infty}\bigg) \bigg(\sum\limits_{|\beta_3|\leqslant N-1}\|\partial^{\beta_3}g\|_{L^2}\bigg) \\
& +\bigg(\sum\limits_{|\beta_2|\leqslant\lfloor\frac{N-1}{2}\rfloor}\|\partial^{\beta_2}g\|_{L^\infty}\bigg)\bigg(\sum\limits_{|\beta_4|\leqslant N}\|\partial^{\beta_4}f \|_{L^2}\bigg).
\end{aligned}
$$
\end{lemma}

Recall that $A_0$ is a positive definite matrix. By Lemma \ref{nonlinear lemma} and the Sobolev imbedding theorems, we obtain
$$
\begin{aligned}
& (1-C\varepsilon)\int|\partial^\alpha\bm{U}|^2dx+\int|\partial^\alpha\theta|^2dx+2\int_0^t\int|\nabla\partial^\alpha\theta|^2dxd\tau \\
\lesssim & ~ E_{k,1}(0)+E_{\lfloor k/2+5/2\rfloor,1}^{1/2}(t)\big(E_{k,1}(t)+E_{k,2}(t)\big).
\end{aligned}
$$

Thus we get
\begin{equation}\label{Key Energy estimate k-1}
E_{k,1}(t)\lesssim E_{k,1}(0)+E_{\lfloor k/2+5/2\rfloor,1}^{1/2}(t)\big(E_{k,1}(t)+E_{k,2}(t)\big).
\end{equation}

To estimate $E_{k,2}(t)$, we set
$$B_i(\bm{U},\theta)=A_i(\bm{U},\theta)-A_i(\bm{0},0),\quad 0\leqslant i\leqslant 3,$$

then we can rewrite \eqref{Heat-Hydro-Hyperbolic system}, and apply $\partial^\beta ~ (|\beta|\leqslant k-1)$ to get
\begin{equation}\label{Heat-Hydro-Hyperbolic system-2}
\begin{aligned}
& \partial^\beta\bm{U}_t+\sum\limits_{j=1}^3A_j(\bm{0},0)\partial^\beta\partial_j\bm{U} \\
= & \left(
\begin{array}{c}
\partial^\beta a_t+\partial^\beta(\nabla\cdot\bm{u}) \\
\partial^\beta\bm{u}_t+\nabla\partial^\beta a
\end{array}
\right)
=-\partial^\beta\left(B_0\bm{U}_t+\sum\limits_{j=1}^3B_j\partial_j\bm{U}+\bm{F}\right).
\end{aligned}
\end{equation}

Taking inner product of \eqref{Heat-Hydro-Hyperbolic system-2} with the following vector
$$\partial^\beta\bm{V}\triangleq\big(-\partial^\beta(\nabla\cdot\bm{u}),\nabla\partial^\beta a\big),$$

and integrate with respect to $t$, we get
\begin{equation}\label{Key Energy estimate k-2-1}
\begin{aligned}
& \int\partial^\beta\bm{u}\cdot\nabla\partial^\beta adx+\int_0^t\int|\nabla\partial^\beta a|^2-|\partial^\beta(\nabla\cdot\bm{u})|^2+\nabla\partial^\beta a\cdot\nabla\partial^\beta \theta dxd\tau \\
= & ~ \int\partial^\beta\bm{u}_0\cdot\nabla\partial^\beta a_0dx-\int_0^t\int\nabla\partial^\beta a\cdot\partial^\beta(a\nabla\theta)dxd\tau \\
- & ~ \int_0^t\int\partial^\beta\bm{V}\cdot\partial^\beta\left(B_0\bm{U}_t+\sum\limits_{j=1}^3B_j\partial_j\bm{U}\right)dxd\tau.
\end{aligned}
\end{equation}

Taking inner product of \eqref{Heat-Hydro-Hyperbolic system-2} with the following vector
$$\partial^\beta\bm{W}\triangleq\big(\bm{0},-\nabla\partial^\beta\theta\big),$$

we have
\begin{equation}\label{Key Energy estimate k-2-2-1}
\begin{aligned}
- & ~ \int\partial^\beta\bm{u}_t\cdot\nabla\partial^\beta\theta dx-\int\partial^\beta a\cdot\nabla\partial^\beta\theta+|\nabla\partial^\beta\theta|^2dx \\
= & ~ \int\nabla\partial^\beta\theta\cdot\partial^\beta\left(\bm{u}\cdot\nabla\bm{u}+\frac{\theta-a}{1+a}\nabla a\right)dx.
\end{aligned}
\end{equation}

Then we take inner product of the equation of $\partial^\beta\theta$, which is
$$\partial^\beta\theta_t+\partial^\beta(\nabla\cdot\bm{u})-\Delta\partial^\beta\theta=-\partial^\beta\left(\frac{a\Delta\theta}{1+a}+\nabla\cdot(\theta\bm{u})\right),$$

with $\partial^\beta(\nabla\cdot\bm{u})$ to obtain
\begin{equation}\label{Key Energy estimate k-2-2-2}
\begin{aligned}
- & ~ \int\partial^\beta\bm{u}\cdot\nabla\partial^\beta\theta_tdx+\int|\partial^\beta(\nabla\cdot\bm{u})|^2dx-\int\partial^\beta(\nabla\cdot\bm{u})\Delta\partial^\beta\theta dx \\
= & ~ -\int\partial^\beta(\nabla\cdot\bm{u})\partial^\beta\left(\frac{a\Delta\theta}{1+a}+\big(\nabla\cdot(\theta\bm{u})\big)\right)dx.
\end{aligned}
\end{equation}

Adding \eqref{Key Energy estimate k-2-2-1} to \eqref{Key Energy estimate k-2-2-2} and integrating with respect to $t$, we get
\begin{equation}\label{Key Energy estimate k-2-2}
\begin{aligned}
& -\int\partial^\beta\bm{u}\cdot\nabla\partial^\beta\theta dx+\int_0^t\int|\partial^\beta(\nabla\cdot\bm{u})|^2-|\nabla\partial^\beta\theta|^2dxd\tau \\
- & ~ \int_0^t\int\nabla\partial^\beta a\cdot\nabla\partial^\beta\theta dxd\tau-\int_0^t\int\partial^\beta(\nabla\cdot\bm{u})\Delta\partial^\beta\theta dxd\tau \\
= & ~ -\int\partial^\beta\bm{u}_0\cdot\nabla\partial^\beta\theta_0dx+\int_0^t\int\nabla\partial^\beta\theta\cdot\partial^\beta\left(\bm{u}\cdot\nabla\bm{u}+\frac{\theta-a}{1+a} \nabla a\right)dxd\tau \\
- & ~ \int_0^t\int\partial^\beta(\nabla\cdot\bm{u})\partial^\beta\left(\frac{a\Delta\theta}{1+a}+\nabla\cdot(\theta\bm{u})\right)dxd\tau.
\end{aligned}
\end{equation}

Adding \eqref{Key Energy estimate k-2-1} to \eqref{Key Energy estimate k-2-2}, we have
\begin{equation}\label{Key Energy estimate k-2-3}
\begin{aligned}
& \int\partial^\beta\bm{u}\cdot\nabla\partial^\beta(a-\theta)dx+\int_0^t\int|\nabla\partial^\beta a|^2-|\nabla\partial^\beta\theta|^2-\partial^\beta(\nabla\cdot\bm{u})\Delta \partial^\beta\theta dxd\tau \\
= & ~ \int\partial^\beta\bm{u}_0\cdot\nabla\partial^\beta(a_0-\theta_0)dx-\int_0^t\int\nabla\partial^\beta a\cdot\partial^\beta(a\nabla\theta)dxd\tau \\
- & ~ \int_0^t\int\partial^\beta\bm{V}\cdot\partial^\beta\left(B_0\bm{U}_t+\sum\limits_{j=1}^3B_j\partial_j\bm{U}\right)dxd\tau \\
+ & ~ \int_0^t\int\nabla\partial^\beta\theta\cdot\partial^\beta\left(\bm{u}\cdot\nabla\bm{u}+\frac{\theta-a}{1+a}\nabla a\right)dxd\tau \\
- & ~ \int_0^t\int\partial^\beta(\nabla\cdot\bm{u})\partial^\beta\left(\frac{a\Delta\theta}{1+a}+\nabla\cdot(\theta\bm{u})\right)dxd\tau.
\end{aligned}
\end{equation}

Thus we have
\begin{equation}\label{Key Energy estimate k-2-4}
\begin{aligned}
\int_0^t\int|\nabla\partial^\beta a|^2dxd\tau\leqslant & ~ \int_0^t\int|\nabla\partial^\beta\theta|^2+\frac{1}{2}|\partial^\beta(\nabla\cdot\bm{u})|^2+\frac{1}{2}|\Delta \partial^\beta\theta|^2dxd\tau \\
+ & ~ \int\partial^\beta\bm{u}_0\cdot\nabla\partial^\beta(a_0-\theta_0)dx-\int\partial^\beta\bm{u}\cdot\nabla\partial^\beta(a-\theta)dx \\
- & ~ \int_0^t\int\partial^\beta\bm{V}\cdot\partial^\beta\left(B_0\bm{U}_t+\sum\limits_{j=1}^3B_j\partial_j\bm{U}\right)dxd\tau \\
- & ~ \int_0^t\int\partial^\beta(\nabla\cdot\bm{u})\partial^\beta\left(\frac{a\Delta\theta}{1+a}+\nabla\cdot(\theta\bm{u})\right)dxd\tau \\
+ & ~ \int_0^t\int\nabla\partial^\beta\theta\cdot\partial^\beta\left(\bm{u}\cdot\nabla\bm{u}+\frac{\theta-a}{1+a}\nabla a\right)dxd\tau \\
- & ~ \int_0^t\int\nabla\partial^\beta a\cdot\partial^\beta(a\nabla\theta)dxd\tau.
\end{aligned}
\end{equation}

Now go back to \eqref{Key Energy estimate k-2-1}, we have
\begin{equation}\label{Key Energy estimate k-2-5}
\begin{aligned}
\frac{1}{2}\int_0^t\int|\partial^\beta(\nabla\cdot\bm{u})|^2dxd\tau\leqslant & ~ \int_0^t\int\frac{3}{4}|\nabla\partial^\beta a|^2+\frac{1}{4}|\nabla\partial^\beta\theta|^2dxd\tau \\
+ & ~ \frac{1}{2}\int_0^t\int\partial^\beta\bm{V}\cdot\partial^\beta\left(B_0\bm{U}_t+\sum\limits_{j=1}^3B_j\partial_j\bm{U}\right)dxd\tau \\
+ & ~ \frac{1}{2}\int\partial^\beta\bm{u}\cdot\nabla\partial^\beta adx-\int\partial^\beta\bm{u}_0\cdot\nabla\partial^\beta a_0dx \\
+ & ~ \frac{1}{2}\int_0^t\int\nabla\partial^\beta a\cdot\partial^\beta(a\nabla\theta)dxd\tau.
\end{aligned}
\end{equation}

Substituting \eqref{Key Energy estimate k-2-5} into \eqref{Key Energy estimate k-2-4}, by \eqref{Key Energy estimate k-1} we have
\begin{equation}\label{Key Energy estimate k-2}
\begin{aligned}
& \sum\limits_{|\beta|\leqslant k-1}\int_0^t\int|\nabla\partial^\beta a|^2dxd\tau \\
\lesssim & ~ E_{k,1}(t)+E_{k,1}(0)+E_{\lfloor k/2+5/2\rfloor,1}^{1/2}(t)\big(E_{k,1}(t)+E_{k,2}(t) \big) \\
\lesssim & ~ E_{k,1}(0)+E_{\lfloor k/2+5/2\rfloor,1}^{1/2}(t)\big(E_{k,1}(t)+E_{k,2}(t)\big).
\end{aligned}
\end{equation}

Substituting \eqref{Key Energy estimate k-2} into \eqref{Key Energy estimate k-2-5}, similarly by \eqref{Key Energy estimate k-1} we have
\begin{equation}\label{Key Energy estimate k-3}
\begin{aligned}
& \sum\limits_{|\beta|\leqslant k-1}\int_0^t\int|\partial^\beta(\nabla\cdot \bm{u})|^2dxd\tau \\
\lesssim & ~ E_{k,1}(t)+E_{k,1}(0)+E_{\lfloor k/2+5/2\rfloor,1}^{1/2}(t)\big(E_{k,1}(t)+E_{k,2}(t)\big) \\
\lesssim & ~ E_{k,1}(0)+E_{\lfloor k/2+5/2\rfloor,1}^{1/2}(t)\big(E_{k,1}(t)+E_{k,2}(t)\big).
\end{aligned}
\end{equation}

Note that by \eqref{irrotational} and Hodge decomposition, we have
$$\sum\limits_{|\beta|\leqslant k-1}\int_0^t\int|\partial^\beta(\nabla\cdot\bm{u})|^2dxd\tau=\sum\limits_{|\beta|\leqslant k-1}\int_0^t\int|\partial^\beta(\nabla\bm{u})|^2dxd\tau.$$

Adding \eqref{Key Energy estimate k-1}, \eqref{Key Energy estimate k-2} and \eqref{Key Energy estimate k-3}, one has
\begin{equation}\label{Key Energy estimate k}
E_{k,1}(t)+E_{k,2}(t)\lesssim E_{k,1}(0)+E_{\lfloor k/2+5/2\rfloor,1}^{1/2}(t)\big(E_{k,1}(t)+E_{k,2}(t)\big),\quad 1\leqslant k\leqslant s.
\end{equation}

Note that $s>5$, by \eqref{Key Energy estimate 0} and \eqref{Key Energy estimate k} we arrive at
\begin{equation}\label{Energy estimate s}
E_s(t)\triangleq E_{s,1}(t)+E_{s,2}(t)\lesssim E_s(0)+E^{3/2}_s(0)+E^{3/2}_s(t)\leqslant C\big(M^2\varepsilon^2+M^3\varepsilon^3+\varepsilon^{\frac{3}{2}}\big).
\end{equation}

Now we give the proof of Thm\ref{main}.
\begin{proof}
Assume that $t_*<\infty$ in \eqref{energy small assumption}. We take $\varepsilon>0$ small enough, then \eqref{Energy estimate s} gives
$$E_s(t_*)\leqslant 2C\varepsilon^{\frac{3}{2}}<\varepsilon,$$

this contradicts our assumption \eqref{energy small assumption}. Thus we have
$$E_s(t)\leqslant\varepsilon, \quad \forall ~ t\geqslant 0,$$

which completes the proof of Thm \ref{main}.
\end{proof}

\par{\bf Acknowledgements.}
Yi Zhou was supported by Key Laboratory of Mathematics for Nonlinear Sciences, Ministry of Education of China. Shanghai Key Laboratory for Contemporary Applied Mathematics, School of Mathematical Sciences, Fudan University, P.R. China, NSFC (grants No. 11421061). And we sincerely thank Dr. Yi Zhu for her kind help.

\end{document}